\documentclass[a4paper]{amsart}

\usepackage{amsthm,amsfonts,amsmath, amssymb}
\usepackage{color}

\def\excess{\xi}

\newcommand{\conv}{\text{conv}}

\usepackage{geometry}                
\usepackage{anysize}
\marginsize{3cm}{3cm}{3cm}{3cm}

\renewcommand{\baselinestretch}{1.25}
\usepackage[sort&compress,square,comma,numbers]{natbib}

\usepackage{graphicx,float,enumerate}
\usepackage[ruled,boxed,commentsnumbered,norelsize]{algorithm2e}
\usepackage{verbatim}
\usepackage{url}
\usepackage{epstopdf}
\usepackage[notref,notcite,final]{showkeys}

\usepackage[capitalise]{cleveref}

\newtheorem{theorem}{Theorem}

\newtheorem{lemma}[theorem]{Lemma}

\newtheorem{rmk}[theorem]{Remark}

\crefname{conjecture}{Conjecture}{Conjectures}

\crefname{question}{Question}{Questions}

\theoremstyle{definition}

\theoremstyle{remark}

\numberwithin{theorem}{section}

\begin{document}

\title{Polytopes with low excess degree}
\author{Guillermo Pineda-Villavicencio}
\author{Jie Wang}
\author{David Yost}
\address{School of Information Technology, Deakin University, Geelong, Victoria 3220, Australia}
\email{\texttt{guillermo.pineda@deakin.edu.au}}
\address{Centre for Informatics and Applied Optimisation, Federation University, Mt. Helen, Victoria 3350, Australia}
\email{\texttt{wangjiebulingbuling@foxmail.com}}
\email{\texttt{d.yost@federation.edu.au}}

%

\begin{abstract}
We study the existence and structure of $d$-polytopes for which the number $f_1$ of edges  is small compared to the number $f_0$ of vertices. Our results are more elegantly expressed in terms of the excess degree of the polytope, defined as $2f_1-df_0$. We show that the excess degree of a $d$-polytope cannot lie in the range $[d+3,2d-7]$, complementing the known result that values in the range $[1,d-3]$ are impossible. In particular, many pairs $(f_0,f_1)$ are not realised by any polytope. For $d$-polytopes with excess degree $d-2$, strong structural results are known; we establish comparable results for excess degrees $d$, $d+2$, and $2d-6$. Frequently, in polytopes with low excess degree, say at most $2d-6$, the nonsimple vertices all have the same degree and they form either a face or a missing face. We show that excess degree $d+1$ is possible only for $d=3,5$, or $7$, complementing the known result that an excess degree $d-1$ is possible only for $d=3$ or $5$. 
\end{abstract}

\maketitle

\section{Introduction: the Excess theorem}\label{intro}

The combinatorial classification of polytopes under reasonable restrictions is a worthwhile field of research \cite[p. 333]{Ewa96}. Throughout, we denote the dimension of the ambient space by $d$ and the number of vertices and edges of a polytope $P$ by $f_0(P)$ and $f_1(P)$, respectively, or simply by $f_0$ and $f_1$ if $P$ is clear from the context. Briefly, our intention is to exhibit strong structural results for  $d$-dimensional polytopes (henceforth abbreviated as $d$-polytopes) whose excess degree is low (up to $2d-6$). This may lead to a catalogue of such polytopes with few vertices, and the non-existence of polytopes with certain values of $(f_0,f_1)$. For background on polytopes, the reader is referred to \cite{Ewa96}, \cite{PinBook}, or \cite{Zie95}. We begin by motivating and defining the concept of excess degree. 

 The {\it degree} of any vertex is the number of edges incident to it;  this cannot be less than the dimension of the polytope. The {\it excess degree} $\xi$ of a vertex $u$ in a $d$-polytope is defined as $\xi(u)=\deg u-d$; thus a vertex is {\it simple} if its excess degree is zero. The {\it excess} degree of a $d$-polytope $P$, denoted $\xi(P)$, is then defined as  the sum of the excess degrees of its vertices, i.e. $\sum_{u\in {\rm Vert} P}\xi(u)$, where ${\rm Vert}P$ denotes the vertex set of $P$. Hence, a polytope is \textit{simple}, meaning every vertex is simple, if $\xi(P)=0$. A vertex is {\it nonsimple} in a $d$-polytope $P$ if its degree in $P$ is at least $d+1$. A polytope with at least one nonsimple vertex is called {\it nonsimple}. It is easy to see that
$$\xi(P)=2f_1(P)-df_0(P).\eqno{(1)}$$

A trivial but useful observation is the following.
\begin{rmk}\label{even}
$\xi(P)$ must be even if $d$ is even.
\end{rmk}

Considering the class of all $d$-polytopes with $f_0$ vertices, for fixed $d$ and $f_0$, the problem of minimising $\xi(P)$ is  equivalent to the problem of minimising $f_1$; it was this latter problem that led us to this line of research. Here, we focus on the existence and classification of $d$-polytopes with ``low" excess degree, up to $2d-6$. The structure of the nonsimple vertices in such polytopes is quite restricted. In general, they
all have the same degree, and
form the vertex set of a face (or a missing face). Additionally, knowledge of the excess degree implies restrictions on the value of $f_0$.

In slightly more detail, the following summarise our results. Note that the polytopes with excess degree 0 are simply the simple polytopes, which are well studied, so we will not elaborate on  them. 
Assertions (i), (ii), and (iii) are known \cite{PinUgoYosEXC} and are included here for completeness. 
\begin{enumerate}[(i)]
    \item {\bf The Excess theorem} \cite[Theorem 3.3]{PinUgoYosEXC}: The excess degree of any nonsimple $d$-polytope is at least $d-2$.
    \item If a $d$-polytope has excess degree $d-2$, then either there is a vertex with excess degree $d-2$, or there are $d-2$ vertices, each with excess degree 1, that form the vertex set of a simplex face \cite[Theorem 4.10]{PinUgoYosEXC}.
    \item If a $d$-polytope has excess degree $d-1$, then all the nonsimple vertices have the same degree, and either $d=3$ or $d=5$, with the latter case having them as the vertex set of a face \cite[Theorem 4.18]{PinUgoYosEXC}. However, there is almost no restriction on the value of $f_0$ .
    \item If a $d$-polytope has excess degree $d$ and $d\ge7$, then it is (in a sense to be made precise) a certain sum of two simple polytopes which intersect in a common simplex facet, resulting in $d$ nonsimple vertices, each with excess degree 1. Moreover, $f_0$ is either $d+2$, $2d+1$, or $\ge3d$ (\cref{{lem:excess-d-part-2}}, \cref{{thm:excess-d}}).
    \item If a $d$-polytope has excess degree $d+1$, then  either $d=3$, $d=5$, or $d=7$.  However, there is almost no restriction on the value of $f_0$ (\cref{deeplusone}).
    \item For $d\ge 9$, a $d$-polytope with excess degree $d+2$  has $d+2$ vertices; in particular, it is 2-neighbourly (Theorem~\ref{thm-d+2}).
    \item No $d$-polytope has excess degree in the range $[d+3,2d-7]$ (Theorem~\ref{thm-d+3to2d-7}).
    \item If a $d$-polytope has excess degree $2d-6$ and $d\ge9$, then the nonsimple vertices all have the same degree and form the vertex set of a face (\cref{excess2d-6}).
\end{enumerate}

 For polytopes with excess degree $d-2$, we can also show that $f_0$ is either $d+2$, $2d-1$, $2d+1$, or $\ge3d-2$. However the proof is long, and will appear elsewhere.

A typical example of excess degree $d$ is obtained by stacking a vertex on a simplex facet of a simple polytope. In (iv), we assert that all examples are obtained similarly for $d\ge7$. We will see that this is almost true when $d=5$. For $d=3, 4$, or 6, the situation is more complex.

For (v), easy examples exist in dimensions 3, 5, and 7.

It is easy to check that a 2-neighbourly $d$-polytope with $d+2$ vertices has excess degree $d+2$. Thus, (vi) asserts that there are no other examples for $d\ge9$. This is the strongest case of a restriction on the value of $f_0$. In lower dimensions, easy examples exist with arbitrarily large values of  $f_0$.

Part (vii) provides new information for $d\ge11$. For example,  if  $d=12$, then no 12-polytope satisfies $f_1=6f_0+8$. 

In (viii), the nonsimple vertices can form a face of dimension 1, $d-4$, or $d-3$.

The proofs of these results are given from \S4 onwards. In \S2, we offer a new proof of the Excess theorem and additional background information. In \S3, we study the class of semisimple polytopes introduced in \cite{PinUgoYosEXC} and establish strong results about their excess degrees.

For excess degrees beyond $2d-6$, the situation becomes less clear. Consider $M(2,d-2)$, a $(d-2)$-fold pyramid over a square;
it has four simplex facets. If we glue a simplex over one these simplex facets, we obtain a polytope $P$ with excess degree  $2d-2$: there are $d-2$ vertices with excess degree 2, two vertices with excess degree 1, and two simple vertices. The nonsimple simple vertices form a simplex that is interior to $P$; specifically  they do not form a face. Later, we present examples where the subgraph of nonsimple vertices is not even connected.

\section{A new proof of the Excess theorem}
\label{sec:new-excess-theorem}
We give a new proof of the Excess theorem \cite[Thm. 3.3]{PinUgoYosEXC}, via  the solution of Gr\"unbaum's lower bound conjecture for 1-dimensional faces.

Gr\"unbaum \cite[Sec. 10.2]{Gru03} posed a conjecture about the minimal number of $m$-faces of $d$-polytopes with $f_0 \le 2d$, which was recently proved by Xue \cite{Xue20}. The first and third author with Ugon \cite{PinUgoYosLBT} had previously established the cases $m=1$ and $m\ge0.62d$. For some extensions of these results to polytopes with more than $2d$ vertices, see \cite{PinUgoYos2d+2,PinYos22, Xue23}.
In this paper, we are only concerned with the case $m=1$. 

First we introduce the minimising polytopes of Gr\"unbaum's {conjecture}. Denote by $\Delta(m,n)$, where $m,n >0$, the Minkowski sum of an $m$-dimensional simplex and an $n$-dimensional simplex lying in complementary subspaces, or any polytope combinatorially equivalent to it. Note that $\Delta(1,d-1)$ is just a prism based on a $(d-1)$-simplex; we will often refer to these simply as simplicial prisms, or even just as prisms. Denote by $M(k,d-k)$ the $(d-k)$-fold pyramid over the simplicial $k$-prism. 

\begin{theorem}\label{thm:triplexes}\cite[Theorem 7]{PinUgoYosLBT} Let $P$ be a $d$-polytope with $d+k$ vertices, where $1\le k\le d$. Then $P$ has excess degree at least $(k-1)(d-k)$, with equality if and only if $P$ is  $M(k,d-k)$. In particular $P$ has at least $d-k$ nonsimple vertices.
\end{theorem}

Let $H$ be a hyperplane intersecting the interior of $P$ and not containing any vertex of $P$, and denote by $H^+$ and $H^-$ the corresponding closed half-spaces. Then $P^+=H^+ \cap P$ and $P^-=H^- \cap P$ is said to be obtained by \emph{truncation} of $P$. In the case that $H^+$ contains only one vertex $v$ of $P$, the facet $H\cap P$ of $P^-$ (or $P^+$) is called the \emph{vertex figure} of $v$, and is denoted $P/v$.

The next observation will be useful later.

\begin{lemma}\label{simpletruncation}
    Let $H$ be a hyperplane intersecting the interior of $P$ and not containing any vertex of a $d$-polytope $P$. Then the vertices of the facet $H\cap P$ (of $P^-$ and $P^+$) are the intersections with $H$ of the edges of $P$ (which meet $H$). In case either endpoint of such an edge is simple, then the corresponding intersection point is a simplex vertex of the $(d-1)$-polytope $H\cap P$.
\end{lemma}

Part (i) of the following formulation of the Excess theorem will be used several times later.

\begin{theorem}\label{thm:excess} 
Let $P$ be a nonsimple $d$-polytope.
\begin{enumerate}[(i)]
    \item Let $k$ be the excess degree of a nonsimple vertex $v$ in $P$. Then $v$ has at least $d-k-2$ nonsimple neighbours.
    \item The sum of the excess degrees of a given nonsimple vertex and all of its neighbours is at least $ d-2$. 
    \item $\xi(P)\ge d-2$. 
\end{enumerate}
\end{theorem}
\begin{proof}
(i) The conclusion is trivial if $k>d-2$. Assume that $k \le d-3$. The vertex figure $P/v$ of $v$ is a $(d-1)$-polytope with $d+k=d-1+k+1$ vertices, and \cref{thm:triplexes} ensures that  $P/v$ has at least $(d-1)-(k+1)$ nonsimple vertices. Since every simple neighbour of $v$ in $P$ corresponds to a simple vertex in the vertex figure, there are at least $d-k-2$ nonsimple neighbours of $v$.

(ii) Choose a nonsimple vertex $v$ in $P$, and denote by $k$ its excess degree. Again, the conclusion is trivial if $k\ge d-2$. If $k<d-2$, then by (i) the sums of these excess degrees (of $v$ and its neighbours)  is at least $k+(d-k-2)$.

(iii) follows from (ii).
\end{proof}


\section{Semisimple Polytopes}\label{semisimple}

The main purpose of this section is to present a strong version of the Excess theorem for semisimple polytopes, which will be useful in later sections. Recall \cite[p. 2013]{PinUgoYosEXC} that a polytope is called {\it semisimple} if any two distinct facets are either disjoint or intersect in a ridge (i.e. a $(d-2)$-face of a $d$-polytope). The motivation for this definition is the following result, which implies that if a $d$-polytope has two facets whose intersection has dimension in the range $[1,d-3]$, then $\xi(P) \ge d-2$. In other words, the Excess theorem for polytopes which are not semisimple is an easy consequence of \cref{basic}.

\begin{lemma}\label{basic}\cite[Lemma 2.6]{PinUgoYosEXC}
Let $F_1$ and $F_2$ be any two distinct nondisjoint facets of a $d$-polytope
$P$ and denote by $j$ the dimension of $F_1\cap F_2$. Then
\begin{enumerate}[(i)]
    \item  Every vertex in $F_1\cap F_2$ has excess degree at least $d-2-j$.
    \item  The total excess degree of $P$ is at least $max\{ \xi (F_1), \xi (F_2), \xi (F_1\cap F_2)\}+(d-2-j)(j + 1)$.
    \item   If $j\ne d-2$, i.e. $F_1\cap F_2$ is not a ridge of the polytope, then $P$ has excess degree at least $d-2$; in particular, $P$ is not simple.
    \item   If $1\le j\le d-3$, then $P$ has excess degree at least $2d-6$. 
\end{enumerate}
\end{lemma}

\begin{proof}
(Sketch).   The degrees of a given  vertex $v$ in $F_1\cap F_2$ , $F_1 \setminus F_1\cap F_2$ ,
and  $F_2 \setminus F_1\cap F_2$  will be at least $j$, $d-1-j$, and  $d-1-j$, respectively. So the total degree of
such a vertex must be at least $2d-2-j$. 
(ii) There are at least $j + 1$ vertices in $F_1\cap F_2$.
(iii), (iv) follow routinely.
\end{proof}

 A polytope in which every pair of facets intersects in a ridge of the polytope will be called {\it super-Kirkman}. (Note that the definition of super-Kirkman is not quite the same as the definition in \cite[p. 13]{Kle74}.) Every super-Kirkman polytope is semisimple  and every simple polytope is semisimple, but neither implication is reversible. It is worth mentioning that in dimension $2, 3$ or $4$, every semisimple polytope is simple \cite[Lemma 2.7]{PinUgoYosEXC}.

\begin{rmk}\label{remark_sk}
  When $m,n >1$, a pyramid over $\Delta(m,n)$ is super-Kirkman but not simple; when $m$ or $n$ equals 1, $\Delta(m,n)$ is a simplicial prism which is simple but not super-Kirkman.   
\end{rmk}

There are several ways to build up a polytope from another polytope with lower dimension. Given a $d$-polytope $P$ and a proper face $F$, the \emph{wedge} of $P$ over $F$, denoted $W=W(P,F)$ is a $(d+1)$-polytope which is the convex hull of two affinely equivalent facets $P_1$ and $P_2$, both of which are combinatorially equivalent to $P$, and whose intersection $P_1\cap P_2=F$. Moreover the only edges outside $P_1\cup P_2$ are those joining corresponding vertices in  $P_1$ and $P_2$. See \cite[\S2.6]{PinBook} for more details of the construction of wedges, free joins and Cartesian products. 

The following establishes the stability of these properties under some basic constructions, and shows that a prism whose base is a pyramid over $\Delta(m,n), m,n>1$ is semisimple but neither super-Kirkman nor simple. 
 
\begin{lemma}\label{SKconstructions}
\begin{enumerate}[(i)]
    \item A prism whose base is semisimple will again be semisimple. 
     \item A simplex of dimension at least 2 is super-Kirkman.
   \item The free join of two super-Kirkman polytopes is again a super-Kirkman polytope. 
   \item A (multifold) pyramid over a super-Kirkman polytope is again super-Kirkman.
    \item A wedge based on a super-Kirkman polytope over one of its facets is a super-Kirkman polytope.
   \item The Cartesian product of two super-Kirkman polytopes is again a super-Kirkman polytope. In particular $\Delta(m,n)$ is super-Kirkman for $m,n\ge2$.
\end{enumerate}
\begin{proof}
Routine. Note that (iv) is not quite a special case of (iii) because a simplex of dimension 0 or 1 is not super-Kirkman. 
\end{proof}
\end{lemma}

\begin{lemma}\label{vertexfig}
    In a semisimple polytope, every vertex figure is a super-Kirkman polytope.
\end{lemma}
\begin{proof}
    Let $v$ be a vertex in a semisimple $d$-polytope $P$.  $G_1$, $G_2$ be two facets in the vertex figure $P/v$. Since there is a one-one correspondence between faces in the vertex figure and faces of $P$ containing $v$, there are two nondisjoint facets $F_1$, $F_2$ of $P$ which correspond to $G_1$, $G_2$. Since $P$ is semisimple, $\dim(F_1\cap F_2)=d-2$. Hence $\dim(G_1\cap G_2)=d-3$.
\end{proof}

Before continuing, we need a complete classification of simple polytopes with a small number of vertices. One possible way to obtain these polytopes is by truncating a face from a simple polytope. Denote by $J(d)$ the simple polytope obtained by truncating one vertex of a simplicial $d$-prism \cite[p. 2017]{PinUgoYosEXC}; it has $3d-1$ vertices. Like prisms, $J(d)$ has some disjoint facets. 

\begin{rmk}\label{Jremark}
    $J(d)$ is simple but not super-Kirkman.
\end{rmk}

We recall the catalogue of simple $d$-polytope with up to $3d$ vertices \cite[Lemma 2.19]{PinUgoYosEXC}, \cite[Lemma 10]{PrzYos16}.

\begin{lemma}\label{lem:simple} Let $P$ be any simple $d$-polytope.
\begin{enumerate}[(i)]
    \item If $P$ has strictly less than $3d$ vertices, then it is either a simplex, a prism, $\Delta(2,d-2)$, $\Delta(3,3)$, $\Delta(3,4)$ or $J(d)$.
    \item Either $f_0\in\{d+1,2d,3d-3\}$ or $f_0\ge3d-1$.
    \item If a simple $d$-polytope has $3d$ vertices, then $d=2, 4$ or 8. If $d=8$, the polytope is $\Delta(3,5)$.
\end{enumerate} 
\end{lemma}

Later we will need slightly more information about the number of vertices of a simple polytope.

\begin{lemma}\label{f-zero-simple}
   If $P$ is a simple $d$-polytope, then either $f_0\in\{d+1,2d,3d-3,3d-1\}$ or $f_0\ge4d-8$. 
\end{lemma}

\begin{proof}
    This is a routine application of the $g$-theorem, for which we refer to \cite[\S8.6]{Zie95}. This asserts that the $f$-vector of any simplicial polytope is the product of an $M$-sequence and a certain matrix $M_d$.
    
    Noting that the last column of $M_d$ is always
$d+1,d-1,d-3,d-5\ldots$ and  then dualising
informs us that the number of vertices of a simple polytope is
$g_0(d+1)+g_1(d-1)+g_2(d-3)+g_3(d-5)+\ldots$ 
where $g_k$ is an $M$-sequence.

But what is an $M$-sequence? These can be defined
purely combinatorially via the boundary operator. We only need to know the following consequences of the definition:
\begin{enumerate}[(i)]
    \item $g_0=1$;
    \item $g_1$ can be an arbitrary natural number;
    \item if $g_k\le1\ (k\ne0)$, then $g_{k+1}\le g_k$.

\end{enumerate}

If  the sum of the entries of an $M$-sequence is at least four, then the corresponding value of $f_0$ is at least 
$d+1+d-1+d-3+d-5=4d-8$.

If the sum of the entries is at most three, the only possibilities are
\begin{enumerate}[(i)]
    \item $1,0\ldots$
    \item $1,1,0,\ldots$
    \item $1,1,1,0,\ldots$
    \item $1,2,0,\ldots$
\end{enumerate}

and they give the first four values of $f_0$ listed above.
\end{proof}

\begin{lemma}\label{SKsimplex}
A super-Kirkman $d$-polytope with no more than $d+4$ vertices must be a simplex.
\end{lemma}

\begin{proof}
By induction on $d$; this is clear if $d=2$.

For general $d$, fix a polytope $P$ satisfying the hypotheses, and choose a vertex $v$. By the previous lemma, $P/v$ is a super-Kirkman $(d-1)$-polytope with no more than $d-1+4$ vertices. By induction $P/v$ is a simplex, so $v$ is a simple vertex. Since $v$ was arbitrary, $P$ is a simple polytope. Then \cref{lem:simple}  informs us that $P$ is either a simplex or $\Delta(1,2)$ or $\Delta(1,3)$. However no prism is super-Kirkman.
\end{proof}

\begin{theorem}\label{thm:semisimplexcess}
Let $P$ be a semisimple $d$-polytope which is not simple. Then 
\begin{enumerate}[(i)]
    \item  Every nonsimple vertex has excess degree at least 4.
    \item  If $v$ is a nonsimple vertex whose neighbours are all simple, then either $\xi(v)=2d-6$, or $\xi(v) \ge 3d-12$.
    \item  If $P$ has exactly $k$ nonsimple vertices, then 
    each nonsimple vertex has excess degree at least $2(d-k-2)$ and so $f_0(P)\ge 3d-2k-3$.
    \item  If $P$ has a unique nonsimple vertex, then its excess degree is either $2d-6$, or $\ge3d-12$. 
    \item  If $P$ has two or more nonsimple vertices, then its  excess degree is at least $4d-16$.
\end{enumerate}
\begin{proof} 
The hypothesis implies that $d\ge5$. 

(i) Suppose $v$ is a vertex with excess degree at most 3. Then $P/v$ is super-Kirkman with at most $(d-1)+3$ vertices. So $P/v$ is a simplex  by \cref{SKsimplex} and $v$ would be simple.

(ii)  The corresponding vertex figure is a simple super-Kirkman $(d-1)$-polytope but not a simplex. By \cref{lem:simple}, the vertex figure has either $3(d-1)-3$ vertices, or at least $4(d-1)-8$ vertices. Thus $v$ has either $3d-6$ neighbours, or at least $4d-12$ neighbours.

(iii) We proceed by induction on $k$. We can assume $0<k<d-3$, as there nothing to prove otherwise.
The base case $k=1$ follows easily from (ii). If $v$ is the unique nonsimple vertex, then its degree is either $3d-6$, or $\ge 4d-12$.  For $d\ge6$, $4d-12\ge3d-6$. If $d=5$, then part (i) says that the degree of $v$ is at least $5+4=3d-6>4d-12$. So $v$ has at least $3d-6$ neighbours, and  $P$ has at least $3d-5$ vertices.

Suppose the claim is true for $2, \dots ,k-1$. Suppose that $P$ has exactly $k$ nonsimple vertices, and let $v$ be one of them. By \cref{vertexfig}, $P/v$ is a $(d-1)$-super-Kirkman polytope with $k'\le k-1$ nonsimple vertices. By induction, $P/v$ has at least $3(d-1)-2k'-3\ge3d-2k-4$ vertices, i.e.  $v$ has at least $3d-2k-4$ neighbours. So $\excess(v)\ge 2d-2k-4=2(d-k-2)$ and  $P$ has at least $3d-2k-3$ vertices.

(iv) This is a special case of (ii).

(v) Suppose $P$ has $k\ge2$ nonsimple vertices. If $2\le k\le d-4$, then $P$ has excess degree at least $2k(d-2-k)\ge4d-16$. If $ k\ge d-3$, then $P$ has excess degree at least $4(d-3)>4d-16$.
\end{proof}
\end{theorem}

\begin{rmk}\label{remark}
    A pyramid over $\Delta(2,d-3)$ is a semisimple (in fact super-Kirkman) $d$-polytope with excess degree $2d-6$. 
    In dimensions 5 and 6, it is not the only one. 
    For any integers $s>0$ and $n>d+s$, a family of $d$-polytopes $C(d, n, s)$ was constructed in \cite[\S4]{NPUY20}  with the following properties: $C(d,n,s)$ has $n$ vertices, is 2-neighbourly provided $d\ge5$, has one facet with $d+s$ vertices, and every other facet is a simplex. It follows that each of the dual polytopes $C(d,n,s)^*$ is super-Kirkman with a unique nonsimple vertex. Careful examination of the construction shows that for every value of $n$, $C(5,n,1)^*$ has excess degree 4 and $C(6,n,1)^*$ has excess degree 6.
\end{rmk}

It is clear from \cref{thm:semisimplexcess}(i) that any nonsimple semisimple polytope has at least $d+5$ vertices. An example of such a polytope is a $(d-4)$-fold pyramid over $\Delta(2,2)$; in fact, it is super-Kirkman as well. We can now characterise semisimple polytopes with exactly $d+5$ vertices. 

\begin{lemma}\label{semisimplexamples}
    \begin{enumerate}[(i)]
        \item  The only 2-neighbourly 3-polytope is the simplex. 
        \item  The only super-Kirkman 3-polytope is the simplex. 
        \item  A simple polytope with $d+5$ vertices  is either a heptagon, a cube, a 5-wedge, $\Delta(2,2)$ or the 5-prism $\Delta(1,4)$.
    \end{enumerate}
\end{lemma}
\begin{proof}
(i) This is well known: the easy part of Steinitz's theorem asserts that every 3-polytope has a planar graph, and the complete graph $K_5$ is not planar.

(ii) The dual polytope must be 2-neighbourly, hence a simplex.

(iii)   Thanks to \cref{lem:simple}, any simple  polytope, other than a simplex, has at least $2d$ vertices. Thus $2d\le d+5$, and the only simple polytopes with $d+5$ vertices are as listed.
\end{proof}

\begin{lemma}\label{semisimpledplus5}
    A semisimple polytope with exactly $d+5$ vertices is either simple, 2-neighbourly or pyramidal.
\end{lemma}
\begin{proof}
    Suppose $P$ is neither simple nor 2-neighbourly. Then there is at least one simple vertex and at least one nonsimple vertex. By \cref{thm:semisimplexcess}(i), each vertex is either simple or adjacent to every other vertex. In particular, any  vertex with degree $d$ is adjacent to every nonsimple vertex. This implies that there are at most $d$ nonsimple vertices. 
    
    Suppose $P$ is not pyramidal either. If $P$ has  $k<d$ nonsimple vertices, \cite[theorem 2.1]{PinUgoYosClSmp} asserts that each nonsimple vertex had at least $d-k$ simple non-neighbours, which is impossible. Thus $P$ has exactly $d$ nonsimple vertices. Then there are 5 simple vertices, which are adjacent to each nonsimple vertex, but not to each other. In particular, the 5 simple vertices form an independent set. This is absurd, because
    the $d$ nonsimple vertices lie in a hyperplane, and so their removal would leave a graph with at most 2 components. (Denoting the hyperplane by $H$, this conclusion is clear if $H$ supports a face of $P$. Otherwise $H$ intersects the interior of $P$. Let $H_1,H_2$ denote the two closed half-spaces whose boundary is $H$, and set $P_i=H_i\cap P$. Then $F=H\cap P$ is a facet of each $P_i$. The subgraph of $P_i$ ($i=1,2$) induced by the vertices outside $F$ is connected. In the case that there are no other vertices in $F$, it is clear that we have at most two components. In the case that there are some other vertices in $F$, since each vertex in $F$ is incident to an edge outside $F$, this leaves a connected graph.)
\end{proof}

Finally we are able to characterise the super-Kirkman polytopes with $d+5$ vertices.

\begin{theorem}\label{deeplusfive}
    Any super-Kirkman polytope with $d+5$ vertices is a multifold pyramid over $\Delta(2,2)$. The only other semisimple polytopes with $d+5$ vertices are the simple examples listed in \cref{semisimplexamples}.
\end{theorem}
\begin{proof}
    By \cref{semisimplexamples}(iii), the only simple super-Kirkman $d$-polytope with $d+5$ vertices is $\Delta(2,2)$. If $d\le4$, every vertex figure is super-Kirkman, hence a simplex, hence $P$ is simple. This establishes the base case for an inductive proof.

    If $d\ge5$, we know that $P$ is not simple. If $P$ is 2-neighbourly, then so is every 3-face, and  \cref{semisimplexamples}(i) ensures that every 3-face is a simplex. By induction every vertex figure of $P$ is a multifold pyramid over $\Delta(2,2)$. In particular, every vertex figure contains a quadrilateral face. But every 2-face in a vertex figure arises from truncating a 3-face of $P$, and truncating a simplex cannot produce a quadrilateral. So \cref{semisimpledplus5} forces $P$ to be pyramidal, and the induction  is clear.
\end{proof}

We finish this section with a study of the possible number of vertices of semisimple and super-Kirkman polytopes.

All the examples given by the next result are semisimple as well.

\begin{theorem}\label{SKexamples}
Fix a positive integer $k\ge7$. Then for all $d$ sufficiently large, there are super-Kirkman polytopes of dimension $d$ with $d+k$ vertices.   
\end{theorem}
\begin{proof}
Note that repeatedly forming pyramids over a $d_0$-dimensional super-Kirkman polytope with $d_0+k$ vertices will give us  $d$-dimensional super-Kirkman polytopes with $d+k$ vertices for all $d\ge d_0$. So for each $k$, we only need to establish the existence for one value of $d$.

First suppose $k-1$ is composite, say $k-1=mn$ where $m,n\ge2$. Then $\Delta(m,n)$ is an $(n+m)$-dimensional super-Kirkman polytope with $(m+1)(n+1)=n+m+k$ vertices. In particular, the conclusion holds for all odd $k\ge7$, and for $k=10$.

Next suppose $k$ is even, and $k\ge14$. Then $k-10\ge4=2n$ for some $n\ge2$. The free join of $\Delta(3,3)$ and $\Delta(2,n)$ is a super-Kirkman polytope with dimension $d=6+n+2+1$ and $f_0=16+3(n+1)=d+k.$

The case $k=8$ was essentially resolved by Maksimenko, Gribanov and Malyshev \cite[Section 4]{MGM19}. They found a polytope denoted $P_{6,10,14}$, and showed it is the only example (in any dimension) of a nonpyramidal 2-neighbourly $d$-polytope with $d+8$ facets. It has dimension 6 and 10 vertices. Naturally its dual $Q$ is a 6-dimensional super-Kirkman polytope with 14 vertices. 

Maksimenko et al. also gave a representation of $P_{6,10,14}$ as a 0-1 polytope, so its face lattice is not hard to analyse. It turns out that $Q$ has a facet $F$ (actually four) with nine vertices. By \cref{SKconstructions}(iii), the wedge $W(Q,F)$  is a 7-dimensional super-Kirkman polytope with 14+14-9=7+12 vertices. This establishes the last remaining case, $k=12$.
\end{proof}

In a remarkable tour de force, Maksimenko et al.  \cite{MGM19} actually classified all 2-neighbourly $d$-polytopes with no more than $d+9$ facets. By duality, this gives a classification of all super-Kirkman polytopes with up to $d+9$ vertices. Thus \cref{deeplusfive} is a special case of their result; however our proof of this case is shorter. 

In particular they \cite{MGM19} showed that there are no 2-neighbourly polytopes with $d+6$ facets. By duality there are no super-Kirkman polytopes with $d+6$ vertices; this implies that in a semisimple polytope, there are no vertices with excess degree five. It is easy to verify that the octagon and the 6-prism are the only simple polytopes with $d+6$ vertices. We do not know whether there are any other semisimple polytopes with $d+6$ vertices.


\section{Structure of polytopes with low excess}\label{structure}

Strong structural results for polytopes with excess degree $d-2$ were given in \cite{PinUgoYosEXC}. Here we present some technical results about the structure of polytopes with slightly higher excess degree, generally in the range $[d-1,2d-7]$. We finish the section with a strong classification of polytopes with excess degree $d+2$. Corresponding results for other values of the excess degree are given in subsequent sections.

When we restrict the excess degree to be less than $2d-6$, \cref{basic} ensures that the intersection of any two facets in a $d$-polytope, if nonempty, has dimension either 0, $d-3$ or $d-2$, i.e.  either a single vertex, a subridge or a ridge of the polytope.

\begin{lemma}\label{notvertex}
    Let $P$ be a $d$-polytope with excess degree in the range $[d-1, 2d-6]$. Then no two facets intersect in a single vertex.
\end{lemma}
\begin{proof} Suppose the contrary, let $\{v\}$ be the intersection of two facets. Then $v$ has excess degree at least $d-2$. If $v$ has only simple neighbours, then its vertex figure is a simple $(d-1)$-polytope with between $2d- 1$ and $3d-7$ vertices, which is impossible by \cref{lem:simple}(iv).

So we suppose that there is another nonsimple vertex, say $w$, adjacent to $v$; let $k$ be its excess degree. Note that $k\le d-5$. By \cref{thm:excess}(i), $w$ is adjacent to at least $d-k-2$ nonsimple vertices, including $v$. Hence, the total excess degree of $P$ is at least $k + (d-k-2-1) + (d-2) = 2d-5$, contradicting our hypotheses.
\end{proof}

\begin{lemma}\label{vertex-simple-in-facet}
Let $P$ be a $d$-polytope, $v$  a nonsimple vertex which is simple in  a facet $F$. 
\begin{enumerate}[(i)]
    \item There is another facet $F'$ also containing $v$ such that $F\cap F'$ is not a ridge of the polytope.
    \item Suppose further that $P$ has excess degree in the range $[d-1, 2d-7]$. Then there is another facet $F'$ also containing $v$ such that $F\cap F'$ is a subridge of the polytope.
\end{enumerate}
\end{lemma}
\begin{proof}
(i) Since $v$ is simple in $F$, the vertex $v$ is contained in exactly $d-1$ $(d-2)$-faces within $F$. These correspond to $d-1$ facets containing $v$ and intersecting $F$ in a ridge of the polytope. But $v$ is not simple, so it belongs to at least $d+1$ facets. Thus there is at least one more facet containing $v$, and its intersection with $F$ cannot be a ridge of the polytope.

(ii)  \cref{notvertex} ensures that $F\cap F'$ cannot be a single vertex, so by \cref{basic} it can only be a subridge of the polytope.
\end{proof}

The following lemma is one of the key observations in the discussion.

\begin{lemma}\label{structure1}
Let $P$ be a $d$-polytope with excess degree in the range $[d-1, 2d-7]$. Let
$S$ be a $(d-3)$-face which is the intersection of two facets. Then $S$ is a simplex, both facets are simple, and no vertex in $S$ has any neighbours outside these two facets. In particular, every vertex in $S$ has excess degree exactly 1.   
\end{lemma}
\begin{proof}
Let $S = F_1 \cap F_2$. Observe first that, because $\xi(P)\le2d-7$, $F_1$
and $F_2$ are both simple. (Otherwise the excess degree of the facet would be greater than $d-3$ and each of the vertices in $S$ would contribute excess degree at least 1, so the total excess degree of $P$ would be at least $d-3+d-2=2d-5$.)
Hence $S$ is also simple. Then, $S$ must be a simplex (otherwise $S$ would contain at least $2d-6$ vertices, all nonsimple in $P$). Note that there are at least 3 vertices of $S$ with excess degree exactly 1; otherwise $P$ would have excess degree at least $2(d-4)+2=2d-6$. Denote three of them by $v_1, v_2, v_3$. 

Suppose there is a vertex $u\in S$ with excess degree $> 1$. Then there is a vertex $w\notin F_1\cup F_2$ which is adjacent to $u$. Consider the facets $F$ containing the edge $[u,w]$. The number of such facets should be at least $d-1$. We will show that there are not this many such facets, which is a contradiction.

For any such $F$, $F \cap F_1$ must be either a subridge or a ridge of the polytope. This intersection must contain at least $d-3$ neighbours of $u$ in $F_1$, and so omits at most 2 members of $S$. In particular, $F$ must contain one of  $v_1, v_2, v_3$.

If $F$ omits 2 members of $S$, it must contain both neighbours of $u$ in $F_1\setminus S$, and likewise both neighbours of $u$ in $F_2\setminus S$. But then $F\cap F_1$ is a subridge of the polytope and $F$ is simple. However the degree of $u$ in $F$ is at least $(d-3)+2+2+1>d-1$, and so $u$ is not simple in $F$. This contradiction shows that $F$ omits at most 1 member of $S$.

Consider now the case that $F$ contains $S$. There are two subcases to consider, depending whether each $F\cap F_i$ is a subridge or a ridge of the polytope.

If $F \cap F_1$ is a subridge of the polytope, then $F$ is simple and $F\cap F_1= S$. Since $v_1$ has no neighbours outside $F_1\cup F_2$, all  $d-1$ of its neighbours outside $F_1$ must be in $F_2$, which implies that $F$ and $F_2$ have the same affine hull, and hence are equal. But $w\in F\setminus F_2$, so this is impossible. Likewise  $F \cap F_2$ cannot be a subridge of the polytope either.

Next consider the possibility that $F \cap F_1$ and $F \cap F_2$ are both ridges of the polytope. For each vertex $v\in S$, there are two neighbours outside $F$ (one in each $F_i$), and so the excess degree of $v$ in $F$ is strictly less than its excess degree in $P$. 

Then the excess degree of $F$ is at most $2d-7-(d-2) = d-5$, which implies that $F$ is simple. However the degree of $u$ in $F$ is at least $(d-3)+1+1+1$, so $u$ is not simple in $F$. So this subcase does not arise either.

The remaining case is that $F$
omits one member of $S$. There are three subcases to consider, depending whether each $F\cap F_i$ is a subridge or a ridge of the polytope.

If both  $F\cap F_i$ are subridges of the polytope, then they contain $d-3$ members of $S$ and one member of each  $F_i\setminus S$. But then the degree of $v_1$ in $F$ is $(d-4)+1+1$, which is not enough. So this case does not arise.

If  $F\cap F_1$ is a subridge of the polytope, then $F$ is simple and $F\cap F_1\setminus S$ contains just one vertex. If $F\cap F_2$ is a ridge of the polytope, then $F\cap F_2\setminus S$ contains just 2 vertices. Since $u$ is adjacent to $w$, its degree in $F$ is at least $(d-4)+2+1+1$, meaning $u$ is not simple in $F$. So this case does not arise either.

If both $F\cap F_i$ are ridges of the polytope, then $F$ is the other facet determined by this ridge, which is determined uniquely by the vertex of $S$ omitted from $F$. (Since $F_i$ is simple, any $(d-4)$-face of $S$ will be contained in exactly three $(d-2)$-faces of $F_i$, of which two will contain the other vertex of $S$, and one will not.) There are only $d-3$ such possibilities.

In summary 
there are only $d-3$  facets containing $[u,w]$. This impossibility shows that $u$ cannot have excess degree more than 1.
\end{proof}

Let us say that a collection of edges of a polytope is {\it projectively parallel} if  the lines spanned by them are either all parallel or all concurrent at a single point.

\begin{lemma}\label{lem:simplex-facet}
 Let $P$ be a simple polytope with a simplex facet $F$. 
 \begin{enumerate}[(i)]
     \item If $P$ itself is not a simplex, then the vertices in $F$ all have different neighbours outside $F$.
     \item The edges outside $F$ but incident to the vertices of $F$ are all projectively parallel.
     \item If $F'$ is another simplex facet which meets $F$, then $P$ is a simplex.
 \end{enumerate}
    Consequently a simplicial polytope has no simple edges, unless it is a simplex.
\end{lemma}
\begin{proof}
(i) Suppose two vertices in $F$ have the same neighbour outside $F$. Then the collection $V_0$ of vertices outside $F$ which have a neighbour in $F$ has cardinality strictly less than $d$ and its removal disconnects $F$ from any vertex outside $F\cup V_0$. By Balinski's theorem,  $F\cup V_0$ must be empty, and so $P$ has strictly less than $2d$ vertices. By \cref{lem:simple}, $P$ is a simplex.

(ii) Consider a triangular 2-face $F_2$ of the simplex facet $F$. There is a unique 3-face $F_3$ containing $F_2$ that is not contained in $F$. Each two edges of $F_3$ incident to $F_2$ but outside $F_2$ will be coplanar, so these 3 edges are pairwise coplanar. Hence the lines containing these 3 edges are projectively parallel. This is the same for all the 2-faces of $F$. Considering a chain of triangular 2-faces in $F$ with each successive pair having a common edge,  we see that all such edges are  projectively parallel. 

(iii) By simplicity, the intersection $F\cap F'$ must be a ridge of the polytope. Then all members of this ridge have the same neighbour outside $F$. So (i) is applicable.

The final statement is just the dual of (iii).
\end{proof}

\begin{lemma}\label{structure2}
Let $P$ be a $d$-polytope with excess degree in  the range $[d-1,2d-7]$. Fix two distinct facets $F_0$ and $F_1$ whose intersection $S$ is a subridge of the polytope. Then 
\begin{enumerate}[(i)]
    \item All the nonsimple vertices lie in $F_0\cup F_1$, and they all have excess degree 1.
    \item Moreover if  $v$ is a nonsimple vertex in $F_0\setminus S$ (or in $F_1\setminus S$), then the convex hull of $v$ and $S$ is a simplex ridge of the polytope.
    \item If there are (at least) 2 nonsimple vertices in $F_0\setminus S$, then $F_0$ is a simplex.
\end{enumerate}
\end{lemma}
\begin{proof} 
(i) If there were a nonsimple vertex outside $F_0\cup F_1$, it could not (by \cref{structure1})
be adjacent to any vertices in $S$. By \cref{thm:excess}(ii) the total excess degree of $P$ would be at least $d-2 + d- 2 = 2d - 4$. So every nonsimple vertex belongs to  $F_0\cup F_1$.

(ii) Now let $v$ be a nonsimple vertex in $F_0\setminus S$. By \cref{vertex-simple-in-facet}(ii) and \cref{structure1}, $v$ has excess degree exactly 1, and so is adjacent to at least $d-3$ nonsimple vertices. Outside $S$, there are at most $2d-7-(d-2)=d-5$ nonsimple vertices, so $v$ is adjacent to at least 3 vertices in $S$. By \cref{lem:simplex-facet}(i), the convex hull of $v$ and $S$ is a simplex ridge of the polytope.

(iii) In the case that there are two such nonsimple vertices in $F_0\setminus S$, their removal would disconnect $S$ from any other vertices in $F_0$. By Balinski's theorem, there cannot be any other vertices, so $F_0$ is a simplex facet.
\end{proof} 

When the excess degree is in the range $[d+3, 2d-7]$, \cref{notvertex} says the intersection of two facets is not a single vertex; if the intersection is a subridge of the polytope, \cref{structure2} ensures that these two facets are simplices and since all the nonsimple vertices are in the union of these facets, there are just not enough nonsimple vertices; and the remaining possibility is that the polytope is semisimple, but \cref{thm:semisimplexcess} says this is not possible either. So we have the following theorem.

\begin{theorem}\label{thm-d+3to2d-7}
   There is no $d$-polytope with excess degree in the range $[d+3, 2d-7]$. 
\end{theorem}

The corresponding result for excess degree $d+2$ is obtained similarly, under the assumption $d+2\le2d-7$.

\begin{theorem}\label{thm-d+2}
   Let $P$ be a $d$-polytope with excess degree $d+2$, with $d\ge 9$. Then $P$ is a 2-neighbourly polytope with $d+2$ vertices.
   \begin{proof}
       It suffices to show that in the case that the intersection of two facets $F_1\cap F_2=S$ is a subridge of the polytope, then all nonsimple vertices are pairwise adjacent. We can assume that $v_1, v_2 \in F_1$, $v_3, v_4 \in F_2$  are the nonsimple vertices outside this subridge. Suppose to the contrary that $v_i$ $(i=1,2,3,4)$ is adjacent to a vertex $w\notin F_1\cup F_2$, then the removal of $v_1, v_2, v_3, v_4$ will disconnect the graph of $P$ (since vertices in $S$ are only adjacent to vertices in $S$ and $v_i$). And now we have $d+2$ nonsimple vertices that are pairwise adjacent, one can see that there are no other vertices of the polytope.
   \end{proof}
\end{theorem}

For $d=3,4,5,6$ or 8, examples with excess degree $d+2$ and more than $d+2$ vertices are easily found by considering multifold pyramids over suitable 3-polytopes with 6 or 7 vertices \cite[p. 62]{Ewa96}. For $d=7$, $M(4,3)$ is such an example. All of these have some simple vertices, so truncation yields examples with arbitrarily large values for $f_0$.


\section{Polytopes with excess degree $d-1$ or $d+1$}\label{excessdpm1}

Polytopes with excess degree $d-1$ and those with excess degree $d+1$ have one feature in common: they only exist in low dimensions. This was already known in the case $d-1$ \cite[Thm. 4.18]{PinUgoYosEXC}; here we present a new proof, building on the work in the previous sections. Note that in both cases $d$ must be odd.

\begin{theorem}[{\cite[Thm.~4.18]{PinUgoYosEXC}}]\label{deeminusone} If there is a $d$-polytope with excess degree $d-1$, then $d=3$ or 5.
\end{theorem}
\begin{proof}
Let $P$ be such a polytope, and suppose that $d\ge7$. 

Thanks to \cref{semisimple}, $P$ cannot be semisimple. \cref{structure2}(ii) then applies, giving us a simplex ridge $R$, whose vertices are precisely the nonsimple vertices of $P$. Each vertex of $R$ has three neighbours outside. So truncating $R$ will give us a new facet which is simple and has $3(d-1)$ vertices. According to \cref{lem:simple}, this is only possible if $d-1=2,4$ or 8, and when $d-1=8$, this facet must be $\Delta(3,5)$.

So we just need to rule out the case $d=9$. Let $F_1$ and $F_2$ be the two facets intersecting in a $(d-3)$-face $S$, without loss of generality suppose $S\subset R\subset F_2$, and denote by $H$ a hyperplane separating $R$ from the other vertices of $P$. Let $H^+$ be the closed half-space containing $R$, $P'=H^+\cap P$, $F=H\cap P$, $F_1'=F_1\cap H$, $F_2'=F_2\cap H$. Then $F$, $F_1'$ and $F_2'$ are facets of $P'$.  Since $F$ contains only simple vertices, each $F_i'$ must intersect $F$ in a ridge of the polytope. But $F_1'\cap F_2'=\emptyset$ because $F_1\cap F_2=S$ is disjoint from $H$. On the other hand, $\Delta(3,5)$ is super-Kirkman, and cannot have two disjoint 7-faces, contradicting our assumption.
\end{proof}

For a 5-polytope with excess degree 4, if turns out that all nonsimple vertices have the same degree, and they form a face. For a 3-polytope with excess degree 2, it is trivial that  all nonsimple vertices have the same degree, but numerous examples show that they need not form a face. See \cite[Section 4]{PinUgoYosEXC} for details.

\begin{theorem}\label{deeplusone} If there is a $d$-polytope with excess degree $d+1$, then $d=3, 5$ or 7.
\begin{proof}
We will suppose that $d\ge 8$, and reach a contradiction. Since $d-1\le2d-7$,  all the results in  \cref{structure} hold. In particular, we may assume that there are two facets whose intersection is a $(d-3)$-face $S$, $S=F_1 \cap F_2$, that there are exactly 3 nonsimple vertices $v_1, v_2, v_3 \notin S$, that  they each have excess degree exactly 1, and either $F_1$ or $F_2$ is a simplex. Without loss of generality, let $F_2$ be a simplex facet, $v_1 \in F_1$, and $v_2, v_3 \in F_2$. 
		
Next, we show that all the nonsimple vertices are pairwise adjacent, and for each of them, (the lines containing) the unique edges incident to them that are incident to a simple vertex are all projectively parallel. 

Since $v_2$ is a nonsimple vertex in the simple facet $F_2$, $v_2$ is contained in a subridge (of the polytope) which is an intersection of $F_2$ and another facet $G$. But $G$ cannot contain both $v_2$ and $v_3$ (because there is no other nonsimple vertex in $F_1$ other than $v_1$ and those in $S$). So this subridge contains $v_2$ and all the vertices of $S$ except one, say $u$. And $G$ contains $v_1$ and all of its neighbours in $F_1$ except the one that is missing in $S$. Notice that $G$ is simple, so $\conv\{ v_1, v_2, S\setminus u  \}$ is a simplex ridge of the polytope. Hence, $v_1$ is adjacent to $v_2$, and by \cref{structure1}, for $v_1, v_2, S\setminus u_i$, the lines containing the unique edges incident to a given simple vertex are all projectively parallel. A similar argument holds for $v_3$.

Now we have $d+1$ nonsimple vertices which are pairwise adjacent, and for each of them, there is an unique edge incident to them that is incident to a simple vertex and all the lines containing these $d+1$ edges are projectively parallel. 
		
Since $F_2$ is a simplex facet, the $d$ parallel edges toward the same direction, and so are the edges incident to the simplex ridge of the polytope. With respect to projective equivalence, $P$ has a line segment for a summand and then we can take a cross section $H \cap P$ of $P$ such that $H$ passes through these $d+1$ projectively parallel edges without passing through any other vertices of $P$. By \cref{simpletruncation}, $H \cap P$ is a simple $(d-1)$-polytope with $d+1$ vertices. This is not possible unless $d-1\le2$.
\end{proof}
\end{theorem}

There is no 7-polytope with 14 vertices and excess degree 8; details of this will appear elsewhere. Apart from this exception, there are $d$-polytopes with excess degree $d+1$ and $f_0$
being any even number from $d+3$ onwards, for $d=$ 3, 5 or 7. Examples are not hard to construct. (For $d=7$, $M(3,4)$ has 10 vertices, truncating an edge gives an example with 20 vertices and $M(5,2)$ has 12 vertices. Repeatedly truncating simple vertices gives examples with excess degree 8 and all possible values of $f_0$.)


\section{Polytopes with excess degree $d$}\label{excessd}

We have seen that having excess degree $d-2$ imposes strong restrictions on the structure of a $d$-polytope \cite[Thm. 4.10]{PinUgoYosEXC}; in particular, all nonsimple vertices have the same degree.  
In this section, we show that similar restrictions apply to polytopes with excess degree $d$, provided $d\ge5$. For $d=3$ or 4, this is not true; it is easy to find examples showing that 3-polytopes with excess degree 3 and 4-polytopes with excess degree 4 suffer no restrictions on the excess degrees of individual vertices. Once again, the behaviour of low-dimensional polytopes is a poor guide to what happens in general.

Suppose that $P_1$ and $P_2$ are two $d$-polytopes lying on different sides of a common supporting hyperplane, and that their intersection  $F=P_1\cap P_2$ is a facet of both. If in addition their union $P=P_1\cup P_2$ is convex, then $P$ is a $d$-polytope. If  every edge of $F$ is also an edge in $P$, we will call $P$ a {\it graph-connected sum} of $P_1$ and $P_2$, along $F$. More generally, we will call a polytope $F$ a {\it phantom face} of $P$ if it is not a face of $P$, but the vertices of $F$ are all vertices of $P$ and the edges of $F$ are precisely the edges of $P$ whose vertices are both in $F$.

In case $P_1$  and $P_2$ are both simple and $F$ is a simplex, it is easy to see that $P$ has excess degree $d$. We will see that for $d\ge7$, all examples with excess degree $d$ arise in this way. There is a little more variety if $d=5$ or 6, and total chaos for $d=3$ or 4.

A {\it missing face} of a polytope was defined in \cite[Section 2]{AltPer80} as a subset $M$ of the vertices which is not the vertex set of any face, but such that every proper subset of $M$ is the vertex set of a face. The concept of phantom face is more general than the missing face. For example, when $d=4$, a pyramid over a triangular bipyramid is a graph-connected sum, and the convex hull of the apex of the pyramid and the missing triangle in the bipyramid is a phantom face. However it is not a missing face as defined in \cite{AltPer80}.

Likewise graph-connected sum is more general than the connected sum defined in  \cite[p 149]{Eck06}, \cite[p 85]{PinBook},  \cite[Section 3.2]{Ric96}, \cite[p 274]{Zie95} and elsewhere. 
Both can be loosely described as gluing the polytopes $P_1$ and $P_2$ along the common facet $F$; however each definition of the connected sum in these references implicitly or explicitly requires that $F$ be a missing face of $P$. This is too restrictive for our purpose, and would deprive us of the most interesting examples. For example, every $d$-polytope with $d+2$ vertices and excess degree $d$, of which there are $d-2$ combinatorial types \cite[Section 6.1]{Gru03}, can be described as the graph-connected sum of two simplices. Only one of them, the bipyramid over a simplex, is a connected sum.

\begin{lemma}\label{lem:excess-d-part-2}
Let $P$ be a $d$-polytope with excess degree $d$, where $d\ge7$. Then
\begin{enumerate}[(i)]
\item The nonsimple vertices do not form a facet.
\item There exist two simple facets $F_0,F_1$ whose intersection $S$ is a simplex subridge of the polytope, and two more nonsimple vertices $x_0\in F_0\setminus S$ and $x_1\in F_1\setminus S$ such that the convex hulls of $S\cup\{x_0\}$ and of $S\cup\{x_1\}$ are both ridges of $P$.
\item The $d$ nonsimple vertices are pairwise adjacent. Thus every subset of $S$ is the vertex set of a simplex (which may or may not be a face of $P$).
\end{enumerate}
\end{lemma}
\begin{proof} (i) Suppose $F$ is a facet containing all and only the nonsimple vertices. Then $S$ has $d$ vertices, each of which has two neighbours outside $F$, both of them simple. If we truncate $F$, the new polytope will have a simple facet with $2d=2(d-1)+2$ vertices. By \cref{lem:simple}, this is only possible if $2(d-1)+2\ge3(d-1)-3$, i.e. if $d\le6$.

(ii) The existence of $F_0,F_1$,  $x_0$ and $x_1$ follows from \cref{structure2}; we only need to show that $x_0$ and $x_1$ cannot both lie in the same facet. But if they were both in say $F_0$, it would be a simplex facet containing all and only the nonsimple vertices, which is impossible according to part (i).

(iii) We may label the nonsimple vertices of $P$ as $x_0,\ldots,x_{d-1}$ in such a way that  the vertices of $S$ are $x_2,\ldots,x_{d-1}$. By \cref{structure1},  $\{x_0,x_2,\ldots,x_{d-1}\}$ and $\{x_1, x_2,\ldots,x_{d-1}\}$ are the vertex sets of two simplex ridges of $P$. Every subset of these two sets thus generates a simplex face.

\cref{vertex-simple-in-facet}(ii) tells us that there is another facet $F'$ containing $x_0$ such that $F_0\cap F'$ is a subridge $S'$ of $P$. \cref{lem:simplex-facet} informs us that $F'$ contains all but one of the nonsimple vertices, say $x_i$. Clearly $i\ne0$. \cref{structure2}(ii),  $S'$ and $x_i$ generate a simplex ridge of $P$ contained in $F'$. The unique nonsimple vertex in $F_0\setminus F'$ cannot be $x_1$, so $x_0$ and $x_1$ are adjacent.
\end{proof}

\begin{theorem}\label{thm:excess-d}
    For $d\ge7$, every polytope with excess degree exactly $d$ occurs as the graph-connected sum of two simple polytopes along a simplex facet. Therefore the number of vertices of such a polytope is either $d+2$, $2d+1$, or $\ge3d$. 
    \end{theorem}
\begin{proof}
    We know that if $P$ is such a polytope, then there are $d$ vertices $x_1,\cdots,x_d$ with excess degree one, and all other vertices are simple. Denote by $S_i$ the convex hull of $\{x_1,\cdots,x_d\}\setminus \{x_i\}$. We will show that each simplex $S_i$ is contained in the boundary $\partial P$ of $P$, although it is not necessarily a face of $P$.

    Let $\{x_3, \dots, x_{d}\}$ be the vertices of the subridge (of the polytope) which is the intersection of two facets $F_1$ and $F_2$, and let $x_1\in F_1$ and $x_2 \in F_2$ be the other two nonsimple vertices. Notice that for each of the nonsimple vertices, there are two additional edges incident to it, besides the edges incident to the other $x_i$. Label those edges in $F_1$ accordingly as $e_1, e_3, \dots, e_d$, label the edges in $F_2$ accordingly as $e_2', \dots, e_d'$. Then label the remaining edge incident to $x_1$ as $e_1'$, and label the remaining edge incident to $x_2$ as $e_2$. See Figure \ref{parallel}.
        \begin{figure}[H]
	\begin{center}
		\includegraphics[width=0.9\textwidth]{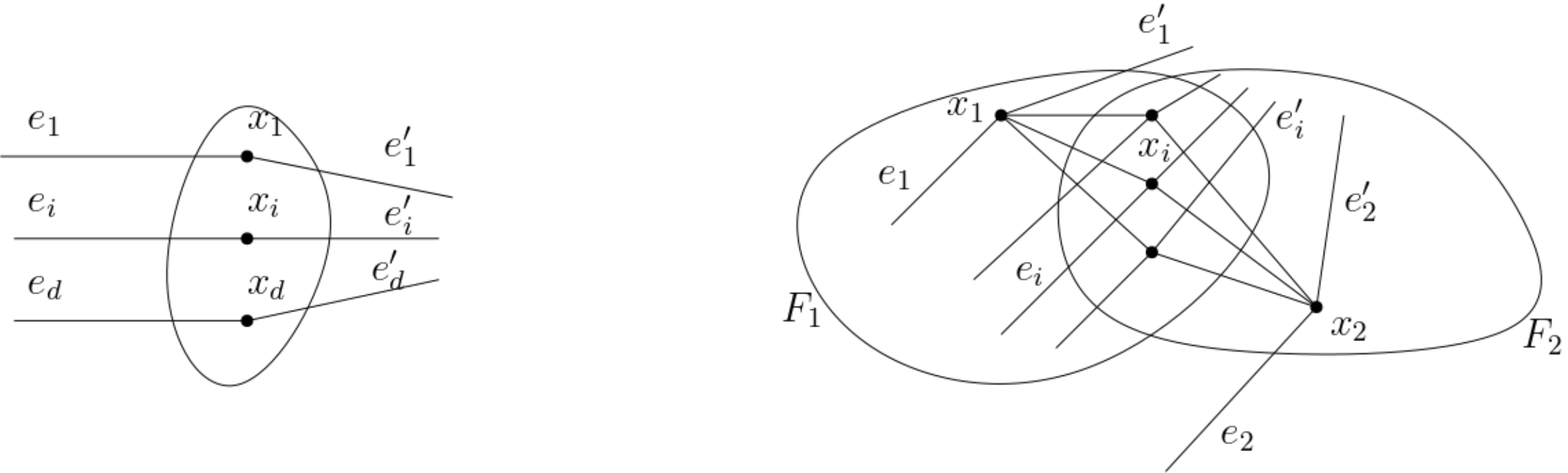}
        \caption{Projectively parallel edges.}
        \label{parallel}
	\end{center}
        \end{figure}
    As $F_1$ is a simple facet and $S_2$ is a $(d-2)$-face therein, by \cref{lem:simplex-facet}, $e_1, e_3,\dots, e_d$ are all projectively parallel, and similarly, $e_2', e_3',\dots, e_d'$ are all projectively parallel. By an argument similar to that in the previous lemma, there is a facet $F'$ containing $x_1$ that intersects with $F_1$ in a $(d-3)$-face and intersects with $F_2$ in a $(d-2)$-face. Since $F'$ has to be simple, and again we have a simplex ridge $S_i$ of $P$ for some $i$, it follows that $e_1'$ is projectively parallel with $e_j', j\ne i$, so $e_1', e_2',\dots, e_d'$ are all projectively parallel. Similarly,  we can show that $e_1, e_2,\dots, e_d$ are all projectively parallel.
	\begin{figure}[H]
        \begin{center}
		\includegraphics[width=0.3\textwidth]{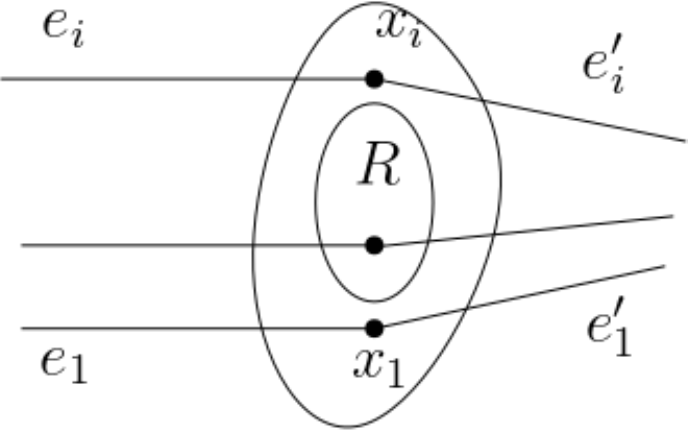}
        \caption{Neighbours of vertices in $R= \conv (S_i \setminus \{x_1\})$.}
	\label{subridgeR}
        \end{center}
        \end{figure}
	We next show that each $S_i$ is contained in a facet, and hence in $\partial P$. It is clear that $S_1\subset F_2$ and $S_2\subset F_1$. Now consider the case $i\ne 1, 2$. Notice that $R$, the convex hull of $ S_i \setminus \{x_1\}$, is a subridge of the polytope. We claim that there is a facet containing both $R$ and $x_1$, i.e. there is a facet containing $S_i$. Then we will have $S_i\subset\partial P$ as required.

 So we consider all the possible facets $F$ containing the subridge $R$, and assume that none of them contain $x_1$. Note that for each $x_j\in R$, the only neighbours of $x_j$ outside $R$ are $x_1$, $x_i$ and (the other endpoints of) $e_j$ and $e_j'$; see Figure \ref{subridgeR}.  
 Any facet containing the subridge $R$ must contain at least two of $x_1$, $x_i$, $e_j$ and $e_j'$. Now there is a unique facet containing $R$, $e_j$ and $x_i$, and a unique facet containing $R$, $e_j'$ and $x_i$. 
 Since there at least three facets containing $R$, there must be one such facet which contains both $e_j$ and $_j'$. 
 
 For any $x_j\in R$ (i.e. $j\ne1,2,i$), the neighbours of $x_j$ in $F\cap F_1$ are the other $d-4$ nonsimple vertices in $R\cap F_1$ and $e_j$, so  $x_j$ has degree only $d-3$ in $F\cap F_1$, which must therefore be a subridge of the polytope. However this intersection also has a simple vertex, namely the other endpoint of $e_j$, contradicting \cref{structure1}. 
 
     So each $S_i\subset\partial P$ after all. Clearly the union $K=\bigcup_{i=1}^dS_i$ is homeomorphic to a $(d-2)$-dimensional sphere. Moreover $K$ is contained in the boundary of $P$, which in turn is homeomorphic to a $(d-1)$-dimensional sphere. The Jordan-Brouwer theorem (\cite[Theorem 6.10.5]{Deo18} or \cite{PerMarKup}) then tells us that $\partial P\setminus K$ has exactly two components, and that $K$ is their common boundary. In particular, any path between these two components must pass through $K$.

    Let $H$ be the unique hyperplane containing $K$, let $H_1$ and $H_2$ be the two open half-spaces containing $H$ and let $S$ be the convex hull of $K$. Then  $S$ is a $(d-1)$-simplex and its facets are the $S_i$. We need to know that there are no vertices of $P$ in $H$, other than those in $K$. Let $v$ be any vertex of $P$ in $H$. Clearly $v$ must have at least one neighbour in $H_1$ and at least one neighbour in $H_2$. These two edges define a path in $\partial P$ from one component to the other. This path must pass through $K$, but $v$ is the only point on either edge which lies in $K$. Thus $v\in K$. 

    If we define $F_i$ as the convex hull of $K$ and the vertices of $P$ which lie in $H_i$, then $P$ is easily seen to be the graph-connected sum of $F_1$ and $F_2$.

The graph-connected sum of two simplices clearly has $d+2$ vertices, as discussed earlier. The graph-connected sum of a simplex and a prism has $2d+1$ vertices and is a capped prism. Otherwise both the simple parts have at least $2d$ vertices, giving $P$ at least $3d$ vertices.
\end{proof}

The situation is more complex in lower dimensions. If $d=6$, there are many natural examples with excess degree 6 which are not graph-connected sums. In each such example, the nonsimple vertices all have the same degree, and they form a face of the polytope (either a single vertex, an edge, a triangle, or a 3-prism). This will be discussed in more detail in the next section.

For $d=5$, the wedge $W(J_4,F)$, where $F$ is any pentagonal 2-face of $J_4$, furnishes an example of a polytope with excess degree 5 which is not a graph-connected sum. The next result shows that all such examples have the same form.

The following lemma was proved without statement by Kusunoki and Murai \cite[p. 97]{KusMur19}. We remark that \cref{thm:excess}(i) can be used to shorten its proof.

\begin{lemma}\label{fiveNSV}
    Let $P$ be a 5-polytope with excess degree exactly 5. Then  $P$ must have five nonsimple vertices, each with excess degree 1.
\end{lemma}

\begin{theorem}\label{excess-five}
    Let $P$ be a 5-polytope with excess degree exactly 5. Then either
    \begin{enumerate}[(i)]
        \item $P$ is the graph-connected sum of two simple polytopes along a simplex facet, or
        \item There exist two simple facets $F_0,F_1$ whose intersection is a pentagonal 2-face.
    \end{enumerate}
    Therefore the number of vertices of such a polytope is 7, 11 or any odd number from 15 onwards.
\end{theorem}

\begin{proof}
    Thanks to \cref{fiveNSV}, $P$ must has five nonsimple vertices, each with excess degree 1. By \cref{thm:semisimplexcess}, $P$ cannot be semisimple, so there are two facets $F_0$ and $F_1$ whose intersection $S$ is a subridge of the polytope. (It cannot be an edge or a single vertex, as then the excess degree of its members would be 2 or more.)

    If the subridge is a triangle (simplex), the argument proceeds as in the preceding theorem, and $P$ is a graph-connected sum as in case (i); the value of $f_0$ will be 7, 11 or $\ge15$. 
    
    If the subridge is a pentagon,
    then we are in case (ii) and both facets are simple. If a simple 4-polytope has a pentagon face, it must have at least 11 vertices. It follows that $f_0\ge17$.

    We rule out the case that $S$ is a quadrilateral. The argument used in \cref{vertex-simple-in-facet}(i) gives us another facet $F$ whose intersection with $F_0$ is not a ridge of the polytope, and so must be a subridge $S'$ of $P$. The intersection $S\cap S'$ can only be an edge, so $S'$ contains just three nonsimple vertices, returning us to case (i).
\end{proof}

This gives us a simpler proof of the nonexistence of a 5-polytope with 13 vertices and 35 edges \cite{KusMur19,PinUgoYosEXC}.
    
For $d=4$ and excess degree 4, there is very little structure. There are examples with
\begin{enumerate}[(i)]
    \item four nonsimple vertices which form a simplex facet,
    \item four nonsimple vertices which are mutually adjacent but do not form a face,
    \item one vertex with excess degree 3 and one vertex with excess degree 1 (necessarily adjacent),
    \item three nonsimple vertices which form a triangular face,
    \item three nonsimple vertices with only one edge between them,
    \item two adjacent vertices with excess degree 2,
    \item two nonadjacent vertices with excess degree 2,
    \item one vertex with excess degree 4.   
\end{enumerate}

Likewise for $d=3$ and excess degree 3, there is very little structure. It is easy to find examples in which  the nonsimple vertices have the same or different degrees and do or do not form a face.


\section{Polytopes with excess $2d-6$}

If $d=3$, 4 or 5, then $2d-6=0, d-2$ or $d-1$ respectively, and these cases are well studied.

For $d\ge6$, there are four basic examples of excess degree $2d-6$.

\begin{enumerate}[(i)]
    \item $M(3,d-3)$, which has $d+3$ vertices, of which $d-3$ each have excess degree 2, forming a simplex $(d-4)$-face,
    \item $M(d-2,2)$, which has $2d-2$ vertices, 2 of which have excess degree $d-3$, 
    \item a prism whose base is $M(2,d-3)$, which has $2d+2$ vertices, of which $2d-6$ each have excess degree 1, forming a prism subridge,
    \item a pyramid over $\Delta(2,d-3)$, which has $3d-5$ vertices, one of which has excess degree $2d-6$.
\end{enumerate}

In each of these cases, the nonsimple vertices all have the same degree, and they form a face. We will see that all examples in sufficiently high dimension have one of these forms. More precisely, we give a characterisation of polytopes with excess degree exactly $2d-6$, in terms of their face figure. 

Let $F$ be a $k$-face of a $d$-polytope $P$ and let $\psi$ be the anti-isomorphism from $P$ onto its dual. The {\it face figure} $P/F$ is a $(d-k-1)$-polytope that is dual to the restriction of $\psi$ to $(\mathcal{F}(\psi(F)), \subset)$. When $F$ is a vertex, it is just the vertex figure.
 
\begin{theorem}\label{excess2d-6}
     Let $P$ be a polytope with excess degree exactly $2d-6$. Then when $d\ge 9$,
\begin{enumerate}[(i)] 
     \item If there is only one nonsimple vertex $v$, then $v$ has excess degree $2d-6$, and $P/v$ is $\Delta(2,d-3)$.
     \item  There do not exist two facets $F_1$, $F_2$ such that $F_1 \cap F_2$ is a single vertex.
     \item  If the intersection of two facets $F_1 \cap F_2$ is a line segment $[v_1, v_2]$, then $v_1, v_2$ have excess degree exactly $d-3$, the other vertices are all simple, and the face figure $P/[v_1, v_2]$ is a $(d-2)$-prism. 
     \item  If the intersection of two facets $F_1 \cap F_2$ is a $(d-4)$-face $K$, then $K$ is a $(d-4)$-simplex, and each vertex in $K$ has excess degree exactly 2. The face figure $P/K$ is a 3-prism.
     \item  If the intersection of two facets $F_1 \cap F_2$ is a $(d-3)$-face $S$, then $S$ is a $(d-3)$-prism, and each vertex in $S$ has excess degree exactly 1. The face figure $P/S$ is a quadrilateral. 
\end{enumerate}    

     \begin{proof}
     (i) If $v$ is the only nonsimple vertex, then $v$ has excess degree $2d-6$, and the vertex figure $P/v$ is a simple $(d-1)$-polytope with $3d-6=3(d-1)-3$ vertices. By \cref{lem:simple}, $P/v$ is $\Delta(2,d-3)$. 
     
     (ii) \cref{notvertex} includes this case.
     
     (iii) The excess degree of $v_1$ (and $v_2$, respectively) is at least $d-3$. Hence $\excess(v_1)=\excess(v_2)=d-3$, and all the other vertices are simple. Consider the face figure $P/[v_1, v_2]$, which is a simple $(d-2)$-polytope. Since $v_1, v_2$ are both simple in $F_1$ and $F_2$, $[v_1, v_2]$ is contained in $d-2$  2-faces of $F_1$ and $F_2$ respectively. Since there are no other edges of $v_1, v_2$ outside $F_1 \cup F_2$, $[v_1, v_2]$ is contained in exactly $2(d-2)$  2-faces of $P$. Hence, $P/[v_1, v_2]$ is a simple $(d-2)$-polytope with $2(d-2)$ vertices, so, by \cref{lem:simple}, $P/[v_1, v_2]$ is a prism.

     (iv) The excess degree of each vertex in $K$ is at least 2, so $K$ must be a $(d-4)$-simplex. This means that every nonsimple vertex is in $K$ and with excess degree exactly 2. Then $F_1, F_2$ are both simple, and $K$ is contained in three $(d-3)$-faces of $F_1$, and three $(d-3)$-faces of $F_2$. So there are exactly six $(d-3)$-faces of $P$ containing $K$. Since all the vertices outside $K$ are simple, $P/K$ is a simple 3-polytope with six vertices, i.e. a 3-prism.

     (v) If $F_1 \cap F_2 = S$ is a $(d-3)$-face, then $S$ could be a simplex or a prism, and $F_1, F_2$ are simple. 
     
     Case 1: $S$ is a simplex. By \cref{thm:excess}(i), some nonsimple vertices have to be adjacent to vertices in $S$, so we assume without loss of generality that $u_1\in S$ has excess degree $>1$. Consider the facets $F$ containing this extra edge of $u_1$. We will use a similar argument as in \cref{structure1}. 
     Suppose first that $F\cap F_1$ is a subridge of the polytope. By routine discussion, one can see that, $F\cap F_1$ cannot be $S$, cannot contain $d-3$ vertices in $S$, nor contain $d-4$ vertices in $S$. Hence $F\cap F_1$ is a ridge of the polytope, and similarly $F\cap F_2$ is also a ridge of the polytope. If $S\subset F$, the excess degree of $F$ is at most $2d-6-(d-2)=d-4$, which is not possible. So, there is exactly one vertex of $S$ missing in $F$, and there are only $d-3$ such possible facets, a contradiction. So, all the edges incident to vertices of $S$ are in $F_1 \cup F_2$, and they all have excess degree exactly 1. Moreover, all the nonsimple vertices are in $F_1\cup F_2$. 
     

     Let $w$ be a nonsimple vertex in $F_1$, then $w$ is a nonsimple vertex in a simple facet, so $w$ is contained in a facet $F$ that intersects with $F_1$ not at a ridge of the polytope. If $F \cap F_1$ is $\{w\}$, or a line segment containing $w$, or a $(d-4)$-face containing $w$, then the excess degree of the polytope is beyond $2(d-3)+1=2d-5$, a contradiction. So $F\cap F_1$ is a subridge of the polytope. Hence, each nonsimple vertex in $F_1, F_2$ has excess degree exactly 1. Notice that $w$ is adjacent to at least two vertices of $S$. This implies that if $F_1$ contains at least two nonsimple vertices, then $F_1$ is a simplex, and so is $F_2$. It follows that $P$ has excess degree $d+2$. So this case does not arise. 
     
     Case 2: $S$ is a $(d-3)$-prism. Every vertex in $S$ is nonsimple and has excess degree exactly 1. The number of $(d-2)$-faces of $P$ containing $S$ is 4, so the face figure $P/S$ is a quadrilateral. 
    		
   \end{proof}
\end{theorem}

For $d=4, 5, 6$ or 8, we have $2d-6=d-2,d-1,d$ or $d+2$ respectively, and examples with excess degree $2d-6$ which are not of the above form are easy to find. We are not aware of any such examples when $d=7$.

\section{chaos}\label{2d-5}

For excess degree above $2d-6$, there seems to be little structure. We expect that excess degree $2d-5$ is only possible when $d=3, 5$ or 7. In these cases, $2d-5=d-2, d$ or $d+2$, so examples do exist. 

If there do exist examples with excess degree $2d-5$ in higher dimensions, the nonsimple vertices would be connected. Indeed \cref{thm:excess}(i) implies that each  component of the subgraph of nonsimple vertices would contribute excess degree at least $d-2$. 

We will see shortly that this is not the case for polytopes with excess degree $2d-4$. Examples with excess degree $2d-4$ are common. The simplest examples are a pyramid over $J(d)$ (with one vertex with excess degree $2d-4$), and a $(d-2)$-fold pyramid over a pentagon (with $d-2$ vertices with excess degree 1). Examples in which the nonsimple vertices have different degrees include a $(d-3)$-fold pyramid over a tetragonal antiwedge ($d-3$ vertices with excess degree 2 and two vertices with excess degree 1) and a pyramid over a $(d-1)$-pentasm (one vertex with excess degree $d-1$ and $d-3$ vertices with excess degree 1).

For our final example, start from $M(d-1,1)$, and truncate one of its simple vertices. The resulting polytope has a unique nonsimple vertex with excess degree $d-2$, and is projectively equivalent to a polytope with a line segment for a summand. The facets of this polytope (on the left of Figure \ref{fig8}) are two simplices, one prism, one $J_{d-1}$, one $M(d-2,1)$ and $d-2$ copies of the $(d-1)$-dimensional version. In Figure \ref{fig8}, a circle represents a simplex face with the indicated number of vertices. Now, add a new vertex $x$ beyond one simplex facet but in the same plane as one of the quadrilateral faces of the prism facet; this 2-face becomes a pentagon in the new polytope. 
In this new polytope (on the right of Figure \ref{fig8}), $d-2$ neighbours of $x$ are nonsimple with excess degree 1, and are not adjacent to $v$, whose excess degree is $d-2$. The total excess degree of this polytope is $2d-4$.

\begin{figure}[H]
\begin{tabular}{cc}
\includegraphics[width=0.32\textwidth]{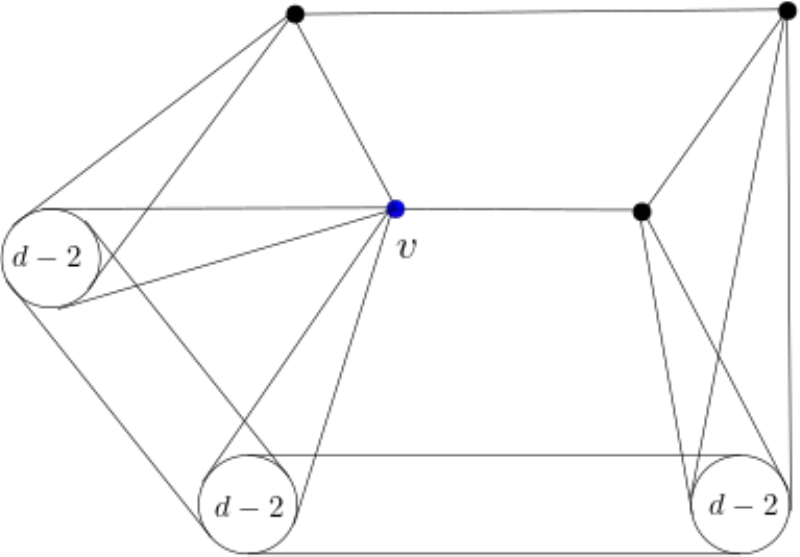}&
\quad \includegraphics[width=0.4\textwidth]{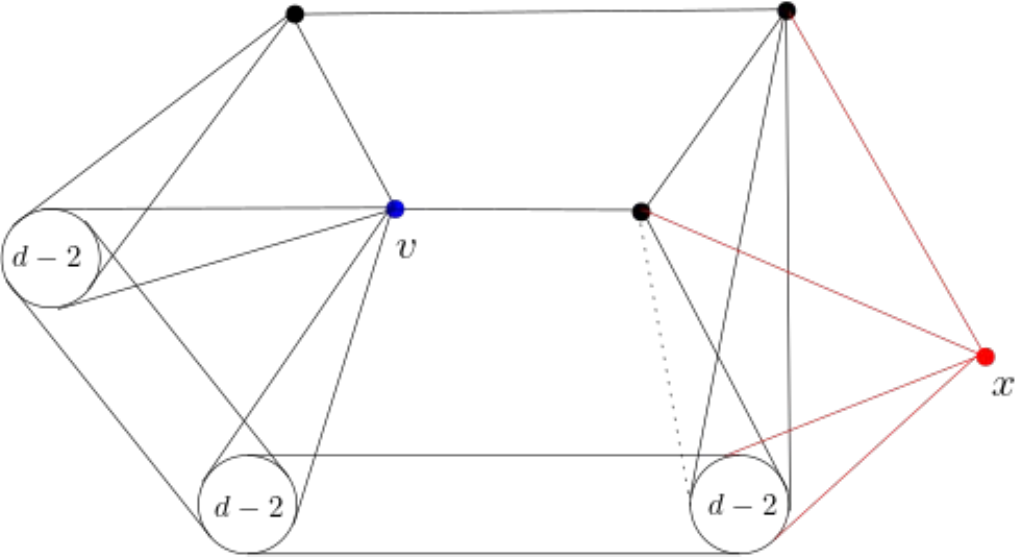} \\
\end{tabular}
\caption{Construction of disconnected subgraph of nonsimple vertices.}
\label{fig8}
\end{figure}


\end{document}